\documentclass[12pt,a4paper]{amsart}
\usepackage[top=35mm, bottom=35mm, left=30mm, right=30mm]{geometry}
\usepackage{mathptmx}
\usepackage{amsmath,amssymb,mathrsfs}
\usepackage[colorlinks=true,citecolor=blue]{hyperref}
\usepackage{xcolor}
\newtheorem{thm}{Theorem}[section]

\newtheorem{lem}[thm]{Lemma}

\theoremstyle{definition}

\newtheorem{rem}[thm]{Remark}
\numberwithin{equation}{section}
\DeclareMathOperator{\supp}{supp}

\newcommand{\eps}{\varepsilon}

\begin{document}
\title{Positive entropy implies chaos along any infinite sequence}
\author{Wen Huang, Jian Li and Xiangdong Ye}

\dedicatory{Dedicated to the memory of Anatoly Mikhailovich Stepin (20.07.1940 -- 07.11.2020)}
 
\address[W.~Huang, X.~Ye]{CAS Wu Wen-Tsun Key Laboratory of Mathematics, School of Mathematical
Sciences, University of Science and Technology of China,
	Hefei, Anhui, 230026, P.R. China}
 \email{wenh@mail.ustc.edu.cn,
	yexd@ustc.edu.cn}
\address[J.~Li]{Department of Mathematics, Shantou University, Shantou, Guangdong 515063, P.R. China}
\email{lijian09@mail.ustc.edu.cn}

\begin{abstract}
Let $G$ be an infinite countable discrete amenable group.
For any $G$-action on a compact metric space $(X,\rho)$, it turns out that
if the action has positive topological entropy, then for any
sequence $\{s_i\}_{i=1}^{+\infty}$ with pairwise distinct elements in $G$
there exists a Cantor subset $K$ of $X$ which is Li-Yorke chaotic along this sequence,
that is, for any two distinct points $x,y\in K$, one has
\[\limsup_{i\to+\infty}\rho(s_i x,s_iy)>0,\ \text{and}\ \liminf_{i\to+\infty}\rho(s_ix,s_iy)=0.\]
\end{abstract}

\maketitle

\section{Introduction}
Throughout this paper, let $G$ be a countable, discrete, infinite, amenable group
with the identity element $e_G$.
By a \emph{$G$-system} we mean a pair $(X,G)$,
where $X$ is a compact metrizable space with a metric $\rho$ and
a continuous action $\Gamma\colon G\times X\to X$
such that $\Gamma(e_G,x)=x$ and $\Gamma(g_1, \Gamma(g_2,x))=
\Gamma(g_1g_2,x)$ for all $g_1,g_2\in G$ and $x\in X$.
As usual, we let $gx=\Gamma(g,x)$ for simplicity.
In the case $G = \mathbb{Z}$,
we can describe the action as a homeomorphism $T\colon X\to X$
 which corresponds to the generator $1$ in $\mathbb{Z}$ and
generates an action $n\mapsto T^n$ through iterations.

Let $S$ be an infinite subset of $G$ and enumerate it as a sequence $\{s_i\}_{i=1}^\infty$,
and $\delta>0$ be a constant.
A pair $(x,y)\in X\times X$ is called \emph{$(S,\delta)$-scrambled}
if
\[\limsup_{i\to+\infty}\rho(s_ix,s_iy)>\delta\quad\text{ and }\quad
\liminf_{i\to+\infty}\rho(s_ix,s_iy)=0.\]
As $S$ is countable, it is easy to see that the above limits
do not depend on the order of elements of $S$.
A subset $K$ of $X$ with at least two points is called \emph{$(S,\delta)$-scrambled} if
every two distinct points $x,y\in K$ form an $(S,\delta)$-scrambled pair.
Following ideas in \cite{LY75}, we say that $(X,G)$ is \emph{$(S,\delta)$-Li-Yorke chaotic}
if there exists an uncountable $(S,\delta)$-scrambled subset of $X$.

Given a $G$-system $(X,G)$, one can define its topological entropy
$h_{top}(X,G)$ lying in $[0,+\infty]$ (see Section 2.2 for details).
One of fundamental questions in topological dynamics is:
if a dynamical system has positive topological entropy,
how complicated 
its orbits can be.
There are many results in this line.
We refer the reader to a recent survey \cite{LY16}.

Here let us mention a few related results.
In \cite{BGKM02} using the measure-theoretical argument Blanchard \textit{et al.} showed that
if a $\mathbb{Z}$-action system $(X,T)$ has positive topological entropy
then there exists a $\delta>0$ and a Cantor $(\mathbb{Z}_+,\delta)$-scrambled subset of $X$.
For a while, people were seeking a proof which does not involve measures. It was done by
Kerr and Li who used a combinatorial method giving an alternative proof and  generalized it to
amenable group actions in \cite{KL07}, and  to sofic group actions in \cite{KL13}.
For a related work, see the recent paper by Li and Rong \cite{LR20}.

The main result of this paper is the following result:

\begin{thm} \label{thm:main-result-1}
If a $G$-system $(X,G)$ has positive topological entropy,
then for any infinite subset $S$ of $G$ there exists a $\delta>0$ and a Cantor $(S,\delta)$-Li-Yorke scrambled set.
\end{thm}

Note that this result is new even when $G=\mathbb{Z}$. After reviewing some auxiliary results of $G$-systems in Section 2,
we will prove the main result in Section 3.

\section{Preliminaries}
In this section, we review some notions and  properties  of $G$-systems which will be used later.
We refer the reader to \cite{EW} and \cite{Gl1}
for textbooks on ergodic theory.

Let $X$ be a compact metric space. A subset $K \subset X$ is called \emph{a Mycielski set}
if it is a union of countably many Cantor sets.
This definition was introduced in~\cite{BGKM02}.
For convenience we restate here a version of Mycielski's theorem (see~\cite[Theorem~1]{My})
which we shall use.
Recall that the \emph{diagonal} of $X\times X$ is
\[\Delta_X=\{(x,x)\in X\times X\colon x\in X\}.\]
\begin{thm}\label{thm:Myc}
	Let $X$ be a perfect compact metric space.
	Assume that $R$ is a dense $G_\delta$ subset of $X\times X$.
	Then  there exists a dense Mycielski subset $K$ of $X$  such that
\[K\times K\subset R\cup \Delta_X.\]
\end{thm}

Let $G$ be a countable, discrete, infinite group.
A sequence $\{F_n\}_{n=1}^{+\infty}$ of non-empty finite subsets  of $G$ is called a \emph{F{\o}lner sequence}
if for every $g\in G$,
\[\lim_{n\to +\infty} \frac{|gF_n\Delta F_n|}{|F_n|}=0,\]
where $|\cdot|$ denotes the cardinality of a set.
It is well known that
$G$ is amenable if and only if it admits a  F{\o}lner sequence.

\subsection{Invariant measures and the disintegration of measures}
Let $(X,G)$ be a $G$-system.
Denote by $\mathcal{B}$ the collection of all Borel subset of $ X$ and
$\mathcal{M}(X)$ the set of all Borel probability measures on $X$.
The \emph{support} of $\mu\in\mathcal{M}(X)$, denoted by $\supp(\mu)$,
is defined to be the set of all points $x\in X$ for which every open
neighborhood of $x$ has a positive measure.

A measure $\mu\in\mathcal{M}(X)$ is called \emph{$G$-invariant} if $\mu(g^{-1}A)=\mu(A)$
for all $g\in G$ and $A\in \mathcal{B}$,
and it is called \emph{ergodic} if it is $G$-invariant and
$\mu(\bigcup_{g\in G}gA)=1$ for any $A\in \mathcal{B}$ with $\mu(A)>0$.
Denote by $\mathcal{M}(X,G)$ (resp.\ $\mathcal{M}^{e}(X,G)$) the set of all $G$-invariant measures
(resp.\ ergodic measures) of $(X,G)$.

Let $\pi\colon (X,\mathcal{B},\mu)\to (Y,\mathcal{D},\nu)$ be a measure-preserving map between Borel probability spaces.
There exists a map
\[\mathbb{E}_{\mu}(\,\cdot\,|Y)\colon
L^1(X,\mathcal{B},\mu)\to L^1(Y,\mathcal{D},\nu),\]
called the \emph{conditional expectation}, such that
every $f\in L^1(X,\mathcal{B},\mu)$, $\mathbb{E}_\mu(f|Y)$
is the unique
element in $L^1(Y,\mathcal{D},\nu)$ satisfying for every $A\in\mathcal{D}$,
\begin{equation}
\int_A \mathbb{E}_\mu(f|Y)d\nu = \int_{\pi^{-1}(A)} f d\mu.
\end{equation}
There also exists a measurable map from $Y$ to $\mathcal{M}(X)$, denoted by $y\mapsto \mu_y$, which satisfies
for every $f\in L^1(X,\mathcal{B},\mu)$, one has $f \in L^1(X,\mathcal{B},\mu_y)$ for $\nu$-a.e. $y\in Y$ and
\begin{equation}
\mathbb{E}_{\mu}(f|Y)(y)=\int_X f\,d\mu_{y}\
\text{for $\nu$-a.e. } y\in Y. \label{eq:meas3}
\end{equation}
We shall write $\mu=\int \mu_yd\nu(y)$ and refer to this as the
\emph{disintegration of $\mu$ with respect to  $(Y,\mathcal{D},\nu)$}.
It is clear that for $\nu$-a.e. $y\in Y$, $\mu_y(\pi^{-1}(y))=1$.

Let $\mathcal{F}$ be a sub-$\sigma$-algebra of $\mathcal{B}$.
There exists a compact metric space $Y$ and a Borel probability measure
$\nu$ on $(Y,\mathcal{D})$, a measure-persevering map
$\pi\colon (X,\mathcal{B},\mu)\to (Y,\mathcal{D},\nu)$ such that
$\pi^{-1}(\mathcal{D})=\mathcal{F} \pmod \mu$.
For every $f\in L^1(X,\mathcal{B},\mu)$, the conditional expectation of $f$ with respect to $\mathcal{F}$,
denoted by $\mathbb{E}_\mu(f|\mathcal{F})$, is the unique element  in $L^1(X,\mathcal{F},\mu)$
such that $\mathbb{E}_\mu(f|\mathcal{F}) \circ \pi=\mathbb{E}_\mu(f|Y)$.

The following lemma is from \cite{R49} (see Lemma 3 in \S4 No. 2).
\begin{lem}\label{lem:Rohlin-measure}
Let $\mu\in \mathcal{M}(X)$ and
$\mu=\int  \mu_y d \nu(y)$ be the disintegration of $\mu$ with respect to  $(Y,\mathcal{D},\nu)$.
Suppose $\mu_y$ is non-atomic for $\nu$-a.e.\ $ y\in Y$.
If $0<r<1$ and $A$ is a measurable subset of $X$ with $\mu_y(A)\leq r$
for $\nu$-a.e.\ $y\in Y$,
then there exists a measurable subset $A'$ of $X$
such that $A \subset A'$ and $\mu_y(A')=r$ for $\nu$-a.e.\ $y\in X$.
\end{lem}

The \emph{relatively independent self-joining of $\mu$ with respect to  $(Y,\mathcal{D},\nu)$} is the Borel probability measure
$\mu\times_Y\mu=\int \mu_y\times \mu_y d\nu(y)$
on $X\times X$ in the sense that
\[\mu\times_Y\mu(A\times B)=\int_Y\mu_y (A)\mu_y(B)d\nu(y)\]
for all $A,B\in\mathcal{B}$.
Let $p\colon X\times X \to X$ be the canonical projection
to the first coordinate.
Then $\pi\circ p\colon (X\times X,\mathcal{B}\otimes\mathcal{B},\mu\times_Y\mu)\to (Y,\mathcal{D},\nu)$ is measure-preserving and
 the disintegration of $\mu\times_Y \mu$ with respect to $(Y,\mathcal{D},\nu)$ is
\[
(\mu\times_Y\mu)_y=\mu_y\times \mu_y.
\]
We have the following variation of Lemma~\ref{lem:Rohlin-measure}.
\begin{lem}\label{lem:Rohlin-measure-2}
Let $\mu\in \mathcal{M}(X)$ and
and
$\mu=\int  \mu_y d \nu(y)$ be the disintegration of $\mu$ with respect to  $(Y,\mathcal{D},\nu)$.
Suppose $\mu_y$ is non-atomic for $\nu$-a.e.\ $ y\in Y$.
If $0<r<1$ and
$B$ is a measurable subset of $X\times X$ with $\mu_y\times \mu_y(B)\leq r$
for $\nu$-a.e.\ $y\in Y$,
then there exists a measurable subset $B'$ of $X\times X$
such that $B\subset B'$ and $\mu_y\times \mu_y(B')=r$ for $\nu$-a.e.\ $y\in Y$.	
\end{lem}

\subsection{Topological entropy and measure-theoretic entropy}
Let $(X,G)$ be a $G$-system.
A \emph{cover} of $X$ is a family of subsets of $X$ whose union is $X$.
Define the \emph{join} of two covers $\mathcal{U}$ and $\mathcal{V}$ of $X$ by
$\mathcal{U}\vee \mathcal{V}=\{U\cap V\colon U\in\mathcal{U}, V\in\mathcal{V}\}$.
Denote by $N(\mathcal{U})$ the number of sets in a subcover of $\mathcal{U}$ of minimal cardinality.

Let $\mathcal{U}$ be an open cover of $X$.
The \emph{entropy of $\mathcal{U}$ with respect to $G$} is defined by
\[h_{top}(G,\mathcal{U})=\lim_{n\to\infty}\frac{1}{|F_n|}
\log N\Bigl(\bigvee_{g\in F_n}g^{-1}\mathcal{U}\Bigr),\]
where $\{F_n\}$ is a F{\o}lner sequence for $G$.
It is well known that the limit exists and is independent of the choice of the F{\o}lner sequence.
The \emph{topological entropy of $(X,G)$} is then defined by
\[h_{top}(X,G)=\sup_{\mathcal{U}}h_{top}(G,\mathcal{U}),\]
where the supremum is taken over all finite open covers of $X$.

A \emph{partition} of $X$ is a family of measurable subsets of $X$
whose elements are pairwise disjoint and the union is $X$.
Given $\mu\in\mathcal{M}(X,G)$,
for a finite partition $\alpha=\{A_1,A_2,\dotsc,A_k\}$ of $X$,
define
\[H_\mu(\alpha)=-\sum_{i=1}^k \mu(A_i)\log\mu(A_i).\]
The \emph{measure-theoretic entropy of $\alpha$ with respect to $\mu$} is defined by
\[h_\mu(G,\alpha)=\lim_{n\to\infty}\frac{1}{|F_n|}
H_\mu\biggl(\bigvee_{g\in F_n}g^{-1}\alpha\biggr),\]
where $\{F_n\}$ is a F{\o}lner sequence for $G$.
It is well known that the limit exists and is independent of the choice of the F{\o}lner sequence.
The \emph{measure-theoretic entropy of $\mu$} is then defined by
\[h_{\mu}(X,G)=\sup_{\alpha}h_{\mu}(G,\alpha),\]
where the supremum is taken over all finite partitions of $X$.
By the well-known variational principle, we have
\[h_{top}(X,G)=\sup_{\mu\in \mathcal{M}(X,G)}h_\mu(X,G)=
\sup_{\mu\in \mathcal{M}^{e}(X,G)}h_\mu(X,G).\]

Let $\mathcal{F}$ be a sub-$\sigma$-algebra of $\mathcal{B}$.
For a finite partition $\alpha$ of $X$,
define
\begin{equation} \label{eq:conditional-entropy}
	H_\mu(\alpha|\mathcal{F})=-
	\int\sum_{A\in\alpha}\mathbb{E}_\mu(1_A|\mathcal{F})\log\mathbb{E}_\mu(1_A|\mathcal{F})d\mu.
\end{equation}
Similarly, we can define the \emph{conditional entropy of a partition $\alpha$ of $X$ given $\mathcal{F}$}
by
\[h_\mu(G,\alpha|\mathcal{F})=\lim_{n\to\infty}\frac{1}{|F_n|}
H_\mu\biggl(\bigvee_{g\in F_n}g^{-1}\alpha|\mathcal{F}\biggr),\]
where $\{F_n\}$ is a F{\o}lner sequence for $G$.

It is well known that if $\alpha$ is a partition of $X$ with $k$-atoms
then $H_\mu(\alpha)\leq \log k$.
We will use the following easy estimation.

\begin{lem}\label{lem:H-mu-alpha2}
Let $\alpha=\{A_1,A_2,\dotsc,A_k\}$ be a partition of $X$.
Then
\[H_\mu(\alpha)\leq \log 2+ (1-\mu(A_1))\log(k-1).\]
\end{lem}
\begin{proof}
Let $\beta=\{A_1,A_1^c\}$. Then
\begin{align*}
H_\mu(\alpha)=H_\mu(\alpha\vee\beta)
=H_\mu(\beta)+H_\mu(\alpha|\sigma(\beta))\leq
 \log 2+ (1-\mu(A_1))\log(k-1),
\end{align*}
where $\sigma(\beta)$ is the $\sigma$-algebra generated by $\beta$.
\end{proof}

\subsection{Pinsker \texorpdfstring{$\sigma$}{sigma}-algebra}
Given a $\mu\in \mathcal{M}^{e}(X,G)$,
\emph{the Pinsker $\sigma$-algebra of $(X,\mathcal{B},\mu,G)$}, denoted by $\mathcal{P}_\mu(G)$,
is the $G$-invariant sub-$\sigma$-algebra of $\mathcal{B}$
generated by all finite partitions $\alpha$ of $X$
with $h_\mu(G,\alpha)=0$.
The Pinsker $\sigma$-algebra $\mathcal{P}_\mu(G)$ corresponds
to a \emph{Pinsker factor} $(Y,\mathcal{P}_\mu(G),\nu,G)$ of $(X,\mathcal{B},\mu,G)$, which plays an important role in
the study of entropy in ergodic theory.

The following result is well known,
see e.g.\ \cite[Corollary 18.20]{Gl1} for $\mathbb{Z}$-actions,
\cite[Lemma 4.3]{HXY15} or \cite[Propsition 3.1]{WZ16} for amenable group actions.
\begin{lem} \label{lem:lambda}
Let $(X,G)$ be a $G$-system and $\mu\in \mathcal{M}^e(X,G)$ with $h_\mu(G)>0$.
Let
\[\mu=\int_Y \mu_y d \nu(y)\]
be the disintegration of $\mu$ with respect to the Pinsker factor $(Y,\mathcal{P}_\mu(G),\nu,G)$ and
$\lambda=\mu \times_{Y}\mu$.
Then $\mu_y$ is non-atomic for $\nu$-a.e.\ $ y \in Y$ and
 $\lambda\in \mathcal{M}^e(X\times X,G)$.
In particular, $\lambda(\Delta_X)=0$.
\end{lem}

The \emph{Pinsker $\sigma$-algebra of $(X,\mu,G)$ given $\mathcal{F}$},
denoted by $\mathcal{P}_\mu(G|\mathcal{F})$,
is the $G$-invariant sub-$\sigma$-algebra of $\mathcal{B}$
generated by all finite partitions $\alpha$ of $X$
with $h_\mu(G,\alpha|\mathcal{F})=0$.

The following result plays a key role 
in our proof of the main result.
Note that it was proved in \cite[Theorem 2.13]{RW00} under the assumption that the action is free,
and in \cite[Theorem 0.1]{D01} in general.
Here we follow \cite[Theorem 6.10]{HYZ11} for this version.
\begin{thm} \label{thm:entorpy-UPE-extension}
Let $(X,G)$ be a $G$-system.
Assume that $\mu\in \mathcal{M}^e(X,G)$, $\alpha$ is a finite partition of $(X,\mu)$
and $\eps>0$.
Then there exists a finite subset $K$ of $G$ such that for every
finite subset $Q$ of $G$ with
 $(QQ^{-1}\setminus \{e_G\})\cap K=\emptyset$, one has
\begin{equation} \label{eq:partition-Pinsker}
\biggl|H_\mu\bigl(\alpha|\mathcal{P}_\mu(G)\bigr) - \frac{1}{|Q|}H_\mu\biggl(\bigvee_{g\in Q}g^{-1}\alpha\biggl|\mathcal{P}_\mu(G)\biggr)
\biggr|<\eps.
\end{equation}
\end{thm}
It is clear that
\[
\frac{1}{|Q|}H_\mu\biggl(\bigvee_{g\in Q}g^{-1}\alpha\biggl|\mathcal{P}_\mu(G)\biggr)
\leq \frac{1}{|Q|}\sum_{g\in Q} H_\mu\bigl(g^{-1}\alpha\bigl|\mathcal{P}_\mu(G)\bigr) = H_\mu\bigl(\alpha|\mathcal{P}_\mu(G)\bigr).
\]
So the term inside the absolute value in the formula \eqref{eq:partition-Pinsker} is non-negative.

Let
\[\mu=\int_Y \mu_y d \nu(y)\]
be the disintegration of $\mu$ with respect to the Pinsker factor $(Y,\mathcal{P}_\mu(G),\nu,G)$.
According to \eqref{eq:meas3} and \eqref{eq:conditional-entropy}, one has
\[
H_\mu\bigl(\beta|\mathcal{P}_\mu(G)\bigr) =
\int_Y H_{\mu_y}(\beta) d\nu(y)
\]
for a finite partition $\beta$ of $X$. Note that for any $y\in Y$,
\[
H_{\mu_y}
\biggl(\bigvee_{g\in Q} g^{-1} \alpha \biggr) \leq\sum_{g\in Q} H_{\mu_y}\bigl(g^{-1} \alpha\bigr),
\]
and for any $g\in G$,
\[
\int_Y H_{\mu_y}(g^{-1}\alpha)d\nu(y)  = H_\mu\bigl(g^{-1}\alpha|\mathcal{P}_\mu(G)\bigr)=H_\mu\bigl(\alpha|\mathcal{P}_\mu(G)\bigr)=
\int_Y H_{\mu_y}(\alpha)d\nu(y).
\]
Thus
\begin{align*}\label{eq:partition-Pinsker2}
 &\hskip0.5cm \int_Y\biggl|\frac{1}{|Q|} \sum_{g\in Q}H_{\mu_y}\bigl(g^{-1} \alpha\bigr)-\frac{1}{|Q|}H_{\mu_y}
 \biggl(\bigvee_{g\in Q} g^{-1}\alpha\biggr)
 \biggr| d\nu(y)\\
 &=\int_Y \big( \frac{1}{|Q|} \sum_{g\in Q}H_{\mu_y}\bigl(g^{-1} \alpha\bigr)-\frac{1}{|Q|}H_{\mu_y}
 \biggl(\bigvee_{g\in Q} g^{-1}\alpha\biggr)
 \big) d\nu(y)\\
 &=\frac{1}{|Q|} \sum_{g\in Q}\int_Y H_{\mu_y}\bigl(g^{-1} \alpha\bigr)d\nu(y)-\frac{1}{|Q|}\int_YH_{\mu_y}
 \biggl(\bigvee_{g\in Q} g^{-1}\alpha
 \biggr) d\nu(y)\\
 &=\int_Y H_{\mu_y}\bigl(\alpha\bigr)d\nu(y)-\frac{1}{|Q|}\int_YH_{\mu_y}
 \biggl(\bigvee_{g\in Q} g^{-1}\alpha
 \biggr) d\nu(y)\\
 &=H_\mu\bigl(\alpha|\mathcal{P}_\mu(G)\bigr) - \frac{1}{|Q|}H_\mu\biggl(\bigvee_{g\in Q}g^{-1}\alpha\biggl|\mathcal{P}_\mu(G)\biggr)\\
 &=\biggl|H_\mu\bigl(\alpha|\mathcal{P}_\mu(G)\bigr) - \frac{1}{|Q|}H_\mu\biggl(\bigvee_{g\in Q}g^{-1}\alpha\biggl|\mathcal{P}_\mu(G)\biggr)
\biggr|.
\end{align*}
Hence we can rewrite the formula \eqref{eq:partition-Pinsker}  as
\begin{equation}\label{eq:partition-Pinsker2}
 \int_Y\biggl|\frac{1}{|Q|} \sum_{g\in Q}H_{\mu_y}\bigl(g^{-1} \alpha\bigr)-\frac{1}{|Q|}H_{\mu_y}
 \biggl(\bigvee_{g\in Q} g^{-1}\alpha\biggr)
 \biggr| d\nu(y)<\eps.
\end{equation}
It should be noticed that the term inside the absolute value in the formula \eqref{eq:partition-Pinsker2} is non-negative,
but this form is convenient for later use.

We will use the following result on the Pinsker $\sigma$-algebra of
$(X\times X,\lambda,G)$.
It  was proved in \cite[Theorem 4]{GTW00} under the assumption that the action is free,
and  in \cite[Theorem 0.4]{D01} in general.
Here we follow \cite[Lemma 4.2]{HXY15} for this version.

\begin{lem}  \label{lem:relative-UPE}
Let $(X,G)$ be a $G$-system and $\mu\in \mathcal{M}^e(X,G)$.
Let
\[\mu=\int_Y \mu_y d \nu(y)\]
be the disintegration of $\mu$ with respect to the Pinsker factor $(Y,\mathcal{P}_\mu(G),\nu,G)$,
$\lambda=\mu \times_{Y}\mu$
and $p\colon X\times X \to X$ is the canonical projection to the first coordinate.
Then
\[\mathcal{P}_\lambda(G|p^{-1}(\mathcal{P}_\mu(G)))=p^{-1}(\mathcal{P}_\mu(G)) \pmod \lambda.\]
\end{lem}

By Lemma~\ref{lem:relative-UPE},
we know that $(X\times X,\lambda,G)$ is  $\mathcal{P}_\mu(G)$-relatively
complete positive entropy (see \cite{GTW00} or \cite{D01})
and the Pinsker $\sigma$-algebra $\mathcal{P}_\lambda(G)$ of $(X\times X,\lambda,G)$
is $p^{-1}(\mathcal{P}_\mu(G)) \pmod \lambda$.
Then
\[\lambda=\int_Y\mu_y\times \mu_y d\nu(y)\]
is the disintegration of $\lambda$ with respect to the Pinsker factor of $(X\times X,\lambda,G)$ and for any finite partition $\beta$ of $X\times X$, we have
\begin{align*}
 H_\lambda(\beta|\mathcal{P}_\lambda(G))
  &=\int_Y H_{\mu_y\times\mu_y}(\beta)d\nu(y).
\end{align*}

\section{Proof of the main result}
In fact, we will prove the following result, which is stronger 
than Theorem~\ref{thm:main-result-1} stated in the introduction.
\begin{thm}\label{thm:Main-result-3}
Assume that $(X,G)$ is a $G$-system, $\mu \in \mathcal{M}^{e}(X,G)$,
$\pi$ is a factor map from $(X,\mathcal{B},\mu,G)$ to its Pinsker factor $(Y,\mathcal{P}_\mu(G),\nu,G)$ and $\mu=\int \mu_y d \nu(y)$
is the disintegration of $\mu$ with respect to $(Y,\mathcal{P}_\mu(G),\nu,G)$.
If $h_\mu(G)>0$, then
for any sequence $\{s_i\}_{i=1}^{+\infty}$ with pairwise distinct elements in $G$, one has the following: for $\nu$-a.e.\ $y\in Y$, there exists a constant
$\delta=\delta(y)>0$ and a dense Mycielski subset $K$ of  $\supp(\mu_y)$
such that $K$ is $(\{s_i\},\delta)$-scrambled, that is,
 for any two distinct points $x_1,x_2\in K$, we have
\begin{equation}
\limsup_{i\to+\infty}\rho(s_i x_1,s_ix_2)\geq\delta, \label{eq:R-delta}
\end{equation}
and
\begin{equation}
\liminf_{i\to+\infty}\rho(s_ix_1,s_ix_2)=0.  \label{eq:P}
\end{equation}
\end{thm}

The main idea of the proof of Theorem~\ref{thm:Main-result-3}
is that after constructing two proper partitions of  $X\times X$
we apply Theorem~\ref{thm:entorpy-UPE-extension}
to show that the collections of pairs satisfying
\eqref{eq:R-delta} and \eqref{eq:P} has a full $\mu_y\times \mu_y$-measure
for $y$ in a large measurable subset of $Y$,
see Lemmas \ref{lem:not-asy} and \ref{lem:proximal-r} below. Then we finish
the proof by applying Theorem~\ref{thm:Myc}.

\medskip

To do so, for every $r>0$, put
\[R_r(\{s_i\})=\Bigl\{(x_1,x_2)\in X\times X\colon
\limsup_{i\to+\infty}\rho(s_ix_1,s_ix_2)\geq r\Bigr\}.\]

\begin{lem}\label{lem:not-asy}
For every $\eps>0$ there exists a $r>0$
and a measurable subset $D$ of $Y$ with $\nu(D)>1-4\eps$
such that $\mu_y\times\mu_y(R_r(\{s_i\}))=1$ for all $y\in D$.
\end{lem}
\begin{proof}
For $t>0$, let
\[\Delta_t=\{(x_1,x_2)\in X\times X\colon \rho(x_1,x_2)<t\}.\]
Let $\lambda=\mu\times_{Y}\mu$.
By Lemma~\ref{lem:lambda}, $\lambda(\Delta_X)=0$.
Then $\lim_{t\to 0}\lambda(\Delta_t)=\lambda(\Delta_X)=0$.
Fix $\eps>0$.
There exists a $r>0$ such that
$\lambda(\Delta_r)<\frac{\eps}{2}$.
Note that
\[\lambda(\Delta_r)=\int_Y \mu_y\times \mu_y(\Delta_r)d\nu(y)
\geq \tfrac{1}{2}\mu(\{y\in Y\colon \mu_y\times \mu_y(\Delta_r)\geq \tfrac{1}{2}\}),\]
so $\nu(V_r)>1-\eps$, where
\[V_r=\{y\in Y\colon 0<\mu_y\times \mu_y(\Delta_r)<\tfrac{1}{2}\}.\]
By Lemma~\ref{lem:Rohlin-measure-2},
there exists a measurable set $\Delta_r\cap(\pi^{-1}(V_r)\times \pi^{-1}(V_r))\subset B\subset X\times X$
with $\mu_y\times\mu_y(B)=\frac{1}{2}$
for $\nu$-a.e.\ $y\in Y$.
Let $A_1=B\cap (\pi^{-1}(V_r)\times \pi^{-1}(V_r))$,
$A_2=(\pi^{-1}(V_r)\times \pi^{-1}(V_r))\setminus A_1$, $A_3=X\times X\setminus(\pi^{-1}(V_r)\times \pi^{-1}(V_r))$,
and  $$\alpha=\{A_1,A_2,A_3\}.$$

Let $Q_1=\{s_1\}$ and suppose that for some $m\in\mathbb{N}$,
we have defined the sets $Q_1,Q_2,\dotsc,Q_{m-1}$.
Now we apply Theorem~\ref{thm:entorpy-UPE-extension} to $(X\times X,\lambda,G)$,
the partition $\alpha$ and the constant $\bigl(\frac{\eps}{2^m}\bigr)^2$,
and we let  $K_m$ be the resulting set.
Choose $Q_m\subset\{s_{m},s_{m+1},\dotsc\}\setminus \bigcup_{i<m}Q_i$ such that
$(Q_mQ_m^{-1}\setminus\{e_G\})\cap K_m=\emptyset$ and $|Q_m|>|Q_{m-1}|$.
Then
\[ \int\biggl\vert
\frac{1}{|Q_m|}H_{\mu_y\times\mu_y}\biggl(\bigvee_{g\in Q_m} g^{-1}\alpha\biggr)-
\frac{1}{|Q_m|}\sum_{g\in Q_m} H_{\mu_y\times\mu_y}(g^{-1}\alpha)\biggr\vert d\nu(y)
<\Bigl(\frac{\eps}{2^m}\Bigr)^2.\]
Let
\[
D_m=\biggl\{y\in Y\colon\biggl\vert \frac{1}{|Q_m|}H_{\mu_y\times\mu_y}
\biggl(\bigvee_{g\in Q_m} g^{-1}\alpha\biggr)
-\frac{1}{|Q_m|}\sum_{g\in Q_m} H_{\mu_y\times\mu_y}(g^{-1}\alpha)\biggr\vert<
\frac{\eps}{2^m}\biggr\}.
\]
Then $\nu(D_m)>1- \frac{\eps}{2^m}$.
Let $D_0=\bigcap_{m=1}^{+\infty} D_m$. Then $\mu(D_0)>1-\eps$.

By \cite[Proposition 5.9]{F81} or \cite[Corollary 5.24]{EW}, for any $g\in G$,
$g\mu_y=\mu_{gy}$ holds for $\nu$-a.e.\ $y\in X$.
Then there exists a $G$-invariant measurable subset $Y_0$ of $Y$ with $\nu(Y_0)=1$
such that
\begin{enumerate}
\item for any $y\in Y_0\cap V_r$, $\mu_y\times\mu_y(A_1)=\frac{1}{2}$;
\item for any $y\in Y_0$ and $g\in G$,
\[H_{\mu_y\times\mu_y}(g^{-1}\alpha)= H_{\mu_{gy}\times\mu_{gy}}(\alpha);\]
\item for any $y\in Y_0$, $\mu_y(\pi^{-1}(y))=1$.
\end{enumerate}
For every $y\in Y_0$ and $g\in G$, we have the following two cases:
\begin{enumerate}
\item[Case 1]   if $gy\in V_r$, one has $\mu_{gy}\times\mu_{gy}(A_1)=\frac{1}{2}$,
$\mu_{gy}\times\mu_{gy}(A_2)=\frac{1}{2}$ and $\mu_{gy}\times\mu_{gy}(A_3)=0$,
then $H_{\mu_y\times\mu_y}(g^{-1}\alpha)=\log 2$;
\item[Case 2]  if $gy\not\in V_r$, one has $\mu_{gy}\times\mu_{gy}(A_1)=0$,
$\mu_{gy}\times\mu_{gy}(A_2)=0$ and $\mu_{gy}\times\mu_{gy}(A_3)=1$,
then $H_{\mu_y\times\mu_y}(g^{-1}\alpha)=0$.
\end{enumerate}
For any $y\in Y$, let $I_m(y)=\{g\in Q_m\colon gy\in V_r\}$ for $m\in \mathbb{N}$.
By conclusions of Cases 1 and 2, we have for any $y\in Y_0$ and $m\in \mathbb{N}$
\[\frac{1}{|Q_m|}\sum_{g\in Q_m} H_{\mu_y\times \mu_y}(g^{-1}\alpha)
=\frac{1}{|Q_m|} |I_m(y)| \log 2\]
and
\[
H_{\mu_y\times \mu_y}\biggl(\bigvee_{g\in Q_m}g^{-1}\alpha\biggr)
=H_{\mu_y\times \mu_y}\biggl(\bigvee_{g\in I_m(y)}g^{-1}\alpha \biggr).
\]
By the definition of $D_m$, we have
\begin{equation}
\biggl|\frac{1}{|Q_m|} H_{\mu_y\times\mu_y}
\biggl(\bigvee_{g\in I_m(y)} g^{-1}\alpha\biggr)-
\frac{1}{|Q_m|}|I_m(y)|\log2\biggr|<\frac{\eps}{2^m}, \label{eq:H-mux-1}
\end{equation}
for every $y\in D_0\cap Y_0$ and every $m\geq 1$.
For every $y\in Y$, define
\[f(y)
=\limsup_{m\to+\infty}\frac{1}{|Q_m|}|I_m(y)|
=\limsup_{m\to+\infty} \frac{1}{|Q_m|} \sum_{g\in Q_m} 1_{V_r}(gy).
\]
Then
\begin{align*}
\int_Y  f(y)d\nu(y)&=\int_Y  \limsup_{m\to+\infty}\frac{1}{|Q_m|} \sum_{g\in Q_m} 1_{V_r}(gy) d\nu(y) \\
&\geq \limsup_{m\to+\infty} \frac{1}{|Q_m|} \sum_{g\in Q_m}\int_Y 1_{V_r}(gy) d\nu(y), \quad \text{ by Fatou's Lemma} \\
&=\nu(V_r)>1-\eps,
\end{align*}
and $\nu(\{y\in Y\colon f(y)>\frac{1}{2}\})>1-2\eps$.
Let
\[D=V_r\cap D_0\cap Y_0\cap \{y\in Y\colon f(y)>\tfrac{1}{2}\}.\]
Then $\nu(D)>1-4\eps$.

Now we show that $D$ is as required. Fix $y\in D$ and $m\in\mathbb{N}$.
Let
\[a_m=\mu_y\times \mu_y\biggl(\bigcap_{g\in I_m(y)} g^{-1}A_1\biggr).\]
If $g \in I_m(y)$, i.e., $gy\in V_r$, then by the conclusion of Case 1 we have
\begin{equation}
  \mu_{gy}\times \mu_{gy}(A_3)=0 \label{eq:mu-A3-0}
\end{equation}
 and
\[ \bigvee_{g\in I_m(y)} g^{-1}\alpha =
\bigvee_{g\in I_m(y)} g^{-1}\{A_1,A_2\} \pmod{\mu_y\times\mu_y}.\]
By \eqref{eq:H-mux-1}, we have
\begin{align*}
\frac{1}{|Q_m|}|I_m(y)|\log2-\frac{\eps}{2^m}&\leq
\frac{1}{|Q_m|}H_{\mu_y\times \mu_y} \biggl(\bigvee_{g\in I_m(y)} g^{-1}\alpha\biggr)\\
&=\frac{1}{|Q_m|}H_{\mu_y\times \mu_y}\biggl(\bigvee_{g\in I_m(y)} g^{-1}\{A_1,A_2\}\biggr)\\
&\leq\frac{1}{|Q_m|}
\Bigl(\log 2+(1-a_m)\log\bigl(2^{|I_m(y)|}-1\bigr)\Bigr),
\quad \text{ by Lemma~\ref{lem:H-mu-alpha2}}\\
&\leq \frac{1}{|Q_m|}\Bigl( \log 2+(1-a_m)|I_m(y)|\log2\Bigr).
\end{align*}
Then
\[
a_m \cdot \frac{1}{|Q_m|}|I_m(x)| \leq \frac{1}{|Q_m|}+\frac{\eps}{2^m\log 2}.
\]
As $\limsup\limits_{m\to+\infty}\frac{1}{|Q_m|}|I_m(y)|=f(y)>\frac{1}{2}$,
we have $\liminf\limits_{m\to+\infty}a_m=0$.

Let
\[Asy_r(\{s_i\},m)=\{(x_1,x_2)\in X\times X\colon
\rho(s_i x_1,s_ix_2)< r,\ \forall i\geq m\}.\]
Note that $Asy_r(\{s_i\},m)$ is increasing as $m$ increases.
Let
\[Asy_r(\{s_i\})=\bigcup_{m=1}^{+\infty} Asy_r(\{s_i\},m).\]
Note that if $\rho(x_1,x_2)<r$, then $(x_1,x_2)\in A_1\cup A_3$.
As $Q_m\subset \{s_{m},s_{m+1},\dotsc\}$,
\begin{align*}
 Asy_r(\{s_i\},m) \subset \bigcap_{g\in Q_m}g^{-1}(A_1\cup A_3).
\end{align*}
So for any $y\in D$,
\begin{align*}
\mu_y\times\mu_y(Asy_r(\{s_i\},m))
&\leq \mu_y\times\mu_y
\biggl(\bigcap_{g\in Q_m}g^{-1}(A_1\cup A_3) \biggr)\\
&\leq \mu_y\times\mu_y
\biggl(\bigcap_{g\in I_m(y)}g^{-1}(A_1\cup A_3) \biggr)\\
&=\mu_y\times\mu_y \biggl(\bigcap_{g\in I_m(y)}g^{-1}A_1 \biggr),
\qquad \text{ by } \eqref{eq:mu-A3-0}\\
&=a_m.
\end{align*}
As $\liminf\limits_{m\to+\infty}a_m\to 0$,
we have $\mu_y\times\mu_y(Asy_r(\{s_i\}))=0$,
because $Asy_r(\{s_i\},m)$ is increasing as $m$ increases.
It is clear that
$X\times X\setminus Asy_r(\{s_i\})\subset R_r(\{s_i\})$.
Thus, $\mu_y\times\mu_y(R_r(\{s_i\}))=1$ for any $x\in D$. This ends the proof.
\end{proof}
For $r>0$, put
\[P_r(\{s_i\})=\Bigl\{(x_1,x_2)\in X\times X\colon \liminf_{i\to+\infty}\rho(s_ix_1,s_ix_2)\leq r\Bigr\}.\]

\begin{lem}\label{lem:proximal-r}
For every $r>0$ and $\eps>0$ there exists a measurable subset $E$ of $X$ with $\mu(E)>1-4\eps$ such that
for any $y\in E$, $\mu_y\times \mu_y(P_r(\{s_i\}))=1$.
\end{lem}
\begin{proof}
Fix $r>0$ and $\eps>0$.
As $\mu_y\times\mu_y(\Delta_r)>0$ for all $y\in Y$, there exists $L\in\mathbb{N}$ such that
$\nu(W_r)>1-\eps$, where $W_r=\{y\in Y\colon \mu_y\times \mu_y(\Delta_r)\geq \frac{1}{L}\}$.
By Lemma~\ref{lem:Rohlin-measure-2}, there exists
  $B_L\subset \Delta_r\cap (\pi^{-1}(W_r)\times \pi^{-1}(W_r))$
such that $\mu_y\times \mu_y(B_L)=\frac{1}{L}$ for $\nu$-a.e.\ $y\in W_r$.
By Lemma~\ref{lem:Rohlin-measure-2} again,
there exist $B_1,\dotsc,B_{L-1}\subset \pi^{-1}(W_r)\times \pi^{-1}(W_r)$ such that
$\{B_1,\dotsc,B_L\}$ is a partition of $\pi^{-1}(W_r)\times \pi^{-1}(W_r)$
and for  $\nu$-a.e.\
$y\in W_r$, $\mu_y\times \mu_y(B_i)=\frac{1}{L}$ for $i=1,2,\dotsc,L$.
Let $B_{L+1}=X\times X\setminus(\pi^{-1}(W_r)\times \pi^{-1}(W_r))$ and
$$\beta=\{B_1,B_2,\dotsc,B_L, B_{L+1}\}.$$

Let $Q_1=\{s_1\}$ and suppose that for some $m\in\mathbb{N}$,
we have define the sets $Q_1,Q_2,\dotsc,Q_{m-1}$.
Now we apply Theorem~\ref{thm:entorpy-UPE-extension} to
$(X\times X,\lambda,G)$, the partition $\beta$ and the constant $\bigl(\frac{\eps}{2^m}\bigr)^2$, and we let $K_m$ be the resulting set.
Choose $Q_m\subset\{s_{m},s_{m+1},\dotsc\}\setminus \bigcup_{i<m}Q_i$ such that
$(Q_mQ_m^{-1}\setminus\{e_G\})\cap K_m=\emptyset$ and $|Q_m|>|Q_{m-1}|$.
Then
\[ \int \biggl| \frac{1}{|Q_m|}H_{\mu_y\times\mu_y}\biggl(\bigvee_{g\in Q_m} g^{-1}\beta\biggr)-
\frac{1}{|Q_m|}\sum_{g\in Q_m} H_{\mu_y\times\mu_y}(g^{-1}\beta) \biggr| d\nu(y)<
\Bigl(\frac{\eps}{2^m}\Bigr)^2.\]
Let
\[
E_m=\biggl\{y\in Y\colon\biggl\vert \frac{1}{|Q_m|}H_{\mu_y\times\mu_y}
\biggl(\bigvee_{g\in Q_m} g^{-1}\beta\biggr)
-\frac{1}{|Q_m|}\sum_{g\in Q_m} H_{\mu_y\times\mu_y}(g^{-1}\beta)\biggr\vert<
\frac{\eps}{2^m}\biggr\}.
\]
Then $\mu(E_m)>1- \frac{\eps}{2^m}$.
Let $E_0=\bigcap_{m=1}^{+\infty} E_m$. Then $\mu(E_0)>1-\eps$.

There exists a $G$-invariant measurable subset $Y_0$ of $Y$ with $\nu(X_0)=1$
such that
\begin{enumerate}
	\item for any $y\in Y_0\cap W_r$,
	$\mu_y\times \mu_y(B_i)=\frac{1}{L}$ for $i=1,2,\dotsc,L$;
	\item for any $y\in Y_0$ and $g\in G$,
	\[H_{\mu_y\times\mu_y}(g^{-1}\beta)= H_{\mu_{gy}\times\mu_{gy}}(\beta);\]
    \item for any $y\in Y_0$, $\mu_y(\pi^{-1}(y))=1$.
\end{enumerate}
For every $y\in Y_0$ and $g\in G$, we have the following two cases:
\begin{enumerate}
	\item[Case 1]  if $gy\in W_r$, one has $\mu_{gy}\times\mu_{gy}(B_i)=\frac{1}{L}$
	for $i=1,2,\dotsc,L$, and $\mu_{gy}\times\mu_{gy}(B_{L+1})=0$,
	then $H_{\mu_y\times\mu_y}(g^{-1}\beta)=\log L$;
	\item[Case 2] if $gy\not\in W_r$, one has $\mu_{gy}\times\mu_{gy}(B_i)=0$
	for $i=1,2,\dotsc,L$, and
	$\mu_{gy}\times\mu_{gy}(B_{L+1})=1$
	then $H_{\mu_y\times\mu_y}(g^{-1}\beta)=0$.
\end{enumerate}
For $y\in Y$, let $I_m(y)=\{g\in Q_m\colon gy\in W_r\}$ for $m\in\mathbb{N}$.
By conclusions of Cases 1 and 2,
we have for any $y\in Y_0$ and $m\in \mathbb{N}$,
\[\sum_{g\in Q_m}H_{\mu_y\times\mu_y}(g^{-1} \beta)=|I_m(y)|  \log L,\]
and
\[H_{\mu_y\times\mu_y}\biggl(\bigvee_{g\in Q_m} g^{-1}\beta\biggr)=
H_{\mu_y\times \mu_y}\biggl(\bigvee_{g\in I_m(y)} g^{-1}\beta\biggr).\]
By the definition of $E_m$, we have
\begin{equation}
\biggl|\frac{1}{|Q_m|} H_{\mu_y\times\mu_y}
\biggl(\bigvee_{g\in I_m(y)} g^{-1}\beta\biggr)-
\frac{1}{|Q_m|}|I_m(y)|\log L\biggr|<\frac{\eps}{2^m}, \label{eq:H-mux-beta-1}
\end{equation}
for every $y\in E_0\cap Y_0$ and every $m\geq 1$.
Similarly as in the proof in Lemma~\ref{lem:not-asy},
the $\nu$-measure of the set
\[Y_1=\Bigl\{y\in Y\colon \limsup_{n\to+\infty} \frac{1}{|Q_m|}|I_m(y)|>\tfrac{1}{2}\Bigr\}\]
is at lease $1-2\eps$.
Let
\[E=W_r\cap E_0\cap Y_0 \cap Y_1.\]
Then $\nu(E)>1-4\eps$.

Now we show that $E$ is as required.
Fix $y\in E$ and $m\geq 1$.
Let
\[C_m=\bigcap_{g\in I_m(y)} g^{-1}(B_1\cup B_2\cup \dotsb\cup B_{L-1})\]
and $b_m=\mu_y\times\mu_y(C_m)$.
If $g\in I_m(y)$, i.e., $gy\in W_r$, then by the conclusion of Case 1
we have
\begin{equation}
\mu_y\times \mu_y(g^{-1}B_{L+1})=0 \label{eq:mu-x-B-L+1}
\end{equation}
 and
\begin{align*}
\bigvee_{g\in I_m(y)} g^{-1}\beta &=
\bigvee_{g\in I_m(y)} g^{-1}(\{B_1,B_2,\dotsc,B_L\}) \pmod{\mu_y\times\mu_y} \\
&=\bigvee_{g\in I_m(y)} g^{-1}(\{B_1,B_2,\dotsc,B_{L-1}\})   \\ 
&\qquad\cup \bigcup_{g_0\in I_m(y)} \biggl(\bigl\{g_0^{-1} B_L\bigr\}\bigvee
\bigvee_{g\in I_m(y)\setminus \{g_0\}} g^{-1}(\{B_1,B_2,\dotsc,B_{L-1},B_L\})\biggr)\\
&=\ (*)\ \cup\ (**)\
\end{align*}
Note that the part $(*)$ can be regarded as a partition of $C_m$  with at most $(L-1)^{|I_m(y)|}$ atoms
and the part $(**)$ as a partition of $X\times X \setminus C_m$ with less than $L^{|I_m(y)|}$ atoms.
By \eqref{eq:H-mux-beta-1}, we have
\begin{align*}
 \frac{1}{|Q_m|}|I_m(y)|\log L-\frac{\eps}{2^m} &\leq
  \frac{1}{|Q_m|} H_{\mu_y\times \mu_y}\biggl(\bigvee_{g\in I_m(y)} g^{-1}\beta\biggr)\\
  &=\frac{1}{|Q_m|}H_{\mu_y\times \mu_y}\bigl(\ (*)\ \cup\ (**)\ \bigr)  \\
&\leq \frac{1}{|Q_m|}\Bigl(\log 2+ b_m|I_m(y)| \log(L-1) \\
&\qquad\qquad\qquad+(1-b_m)|I_m(y)|\log L\Bigr), \quad
\text{by Lemma~\ref{lem:H-mu-alpha2}}.
\end{align*}
Then
\[ b_m\cdot\frac{1}{|Q_m|}|I_m(y)|\leq \frac{1}{\log L-\log(L-1)}\Bigl(\frac{\log2}{|Q_m|}+\frac{\eps}{2^m}\Bigr).\]
As $\limsup\limits_{m\to+\infty}\frac{1}{|Q_m|}|I_m(y)|>\tfrac{1}{2}$,
we have  $\liminf\limits_{m\to+\infty}b_m=0$.

For $m\geq 1$, let
\[J_r(\{s_i\},m)=\{(x_1,x_2)\in X\times X\colon \rho(s_ix_1,s_ix_2)\geq r,\ \forall i\geq m\},\]
and
\[J_r(\{s_i\})=\bigcup_{m=1}^{+\infty} J_r(\{s_i\},m)\]
Note that if $\rho(x_1,x_2)\geq r$ then $(x_1,x_2)\in B_1\cup B_2\cup \dotsb\cup B_{L-1}\cup B_{L+1}$.
As $Q_m\subset \{s_{m},s_{m+1},\dotsc\}$,
\begin{align*}
  J_r(\{s_i\},m)&\subset \bigcap_{g\in Q_m} g^{-1}(B_1\cup B_2\cup \dotsb\cup B_{L-1}\cup B_{L+1})
\end{align*}
and
\begin{align*}
\mu_y\times\mu_y(J_r(\{s_i\},m))&\leq
\mu_y\times\mu_y\biggl(\bigcap_{g\in Q_m} g^{-1}(B_1\cup B_2\cup \dotsb\cup B_{L-1}\cup B_{L+1})\biggr)\\
&=\mu_y\times\mu_y\biggl(\bigcap_{g\in I_m(x)} g^{-1}(B_1\cup B_2\cup \dotsb\cup B_{L-1})\biggr), \qquad\text{ by \eqref{eq:mu-x-B-L+1}}\\
&=b_m.
\end{align*}
Then $\mu_y\times\mu_y(J_r(\{s_i\}))=0$,
as $J_r(\{s_i\},m)$ is increasing as $m$ increases.
It is clear that
$X\times X\setminus J_r(\{s_i\})\subset P_r(\{s_i\})$.
Thus, $\mu_y\times\mu_y(P_r(\{s_i\}))=1$ for any $y\in E$.
This ends the proof of the Claim.
\end{proof}
Now we are ready to prove Theorem~\ref{thm:Main-result-3}
\begin{proof}[Proof of Theorem~\ref{thm:Main-result-3}]
By Lemma~\ref{lem:not-asy},
for every $k\in\mathbb{N}$, there exists a $r_k>0$ and a measurable set $D_k\subset Y$
with $\mu(D_k)>1-\frac{1}{k}$ such that for every $y\in D_k$,
$\mu_y\times \mu_y(R_{r_k}(s_i))=1$.
Let
\[P(\{s_i\})=\Bigl\{(x_1,x_2)\in X\times X\colon \liminf_{i\to+\infty}\rho(s_ix_1,s_ix_2)=0\Bigr\}.\]
It is clear that  $P(\{s_i\})=\bigcap_{k=1}^{+\infty} P_{\frac{1}{k}}(\{s_i\})$.
By Lemma~\ref{lem:proximal-r},
we know that for any $r>0$ and $\eps>0$,
there exists a measurable subset $E'$ of $Y$ with $\nu(E')>1-4\eps$ such that
for any $y\in E'$, $\mu_y\times \mu_y(P_r(\{s_i\}))=1$.
Then
\begin{align*}
	\lambda(P_r(\{s_i\}))&=\int\mu_y\times\mu_y (P_r(\{s_i\}))d\nu(y)\\
	&\geq \int_E\mu_y\times\mu_y (P_r(\{s_i\}))d\nu(y)\\
	&=\nu(E')> 1-4\eps.
\end{align*}
As $\eps>0$ is arbitrary, we have $\lambda(P_r(\{s_i\}))=1$ and thus $\lambda(P\{s_i\})=1$.
This implies that there exists a measurable set $E\subset Y$ with $\nu(E)=1$
such that for every $y\in E$, $\mu_y\times\mu_y(P(\{s_i\}))=1$.
Let
\[D=E\bigcap \biggl(\bigcup_{k=1}^{+\infty}D_k\biggr)\cap\{y\in Y\colon \mu_y\text{ is non-atomic}\}.\]
It is clear that $\nu(D)=1$.

Fix $y\in D$. There exists a $k\in\mathbb{N}$ with $y\in D_k$.
Then $\mu_y\times\mu_y(R_{r_k}(\{s_i\})\cap P(\{s_i\})=1$.
It is easy to see that both $R_{r_k}(\{s_i\})$ and $P(\{s_i\})$ are $G_\delta$ subsets of $X\times X$.
So
\[R_{r_k}(\{s_i\})\cap P(\{s_i\}\cap \bigl( \supp(\mu_y)\times\supp(\mu_y)\bigr)\]
 is a dense $G_\delta$
subset of $\supp(\mu_y)\times\supp(\mu_y)$.
As $\mu_y$ is non-atomic, $\supp(\mu_y)$ is perfect.
Applying Theorem~\ref{thm:Myc},
we get a dense Mycielski
subset $K$ of $\supp(\mu_y)$ such that
$K\times K \subset (R_{r_k}(\{s_i\})\cap P(\{s_i\}) \cup \Delta_X$.
Then $K$ is $(\{s_i\},r_k)$-scrambled, which ends the proof by setting $\delta=r_k$.
\end{proof}

\section{Final remarks}
Finally we make several remarks.
\begin{rem}
For a $\mathbb{Z}$-action system $(X,T)$
there is another approach to the proof of Theorem~\ref{thm:Main-result-3}
instead of using Lemma~\ref{lem:proximal-r}.
Let $\{s_i\}$ be a  sequence of pairwise distinct integers.
Without loss of generality, we can assume that $\{s_i\}$
is an increasing sequence of positive integers.
By \cite[Lemma 3.1]{HLY14}, for $\nu$-a.e.\ $y\in Y$,
\[ \overline{W_{\mathbb{Z}_+}^s(x,T)\cap\supp(\mu_y) }=\supp(\mu_y),\]
where $W_{\mathbb{Z}_+}^s(x,T)=\{y\in X\colon \lim_{i\to\infty}\rho(T^ix,T^y)=0\}$.
In particular,
$$Asy_{\mathbb{Z}_+}(X,T)\cap \bigl(\supp(\mu_y)\times\supp(\mu_y)\bigr)$$
is dense in $\supp(\mu_y)\times\supp(\mu_y)$.
It is clear that $Asy_{\mathbb{Z}_+}(X,T)\subset P(\{s_i\})$, then
\[ P(\{s_i\}) \cap \bigl( \supp(\mu_y)\times\supp(\mu_y)\bigr)\]
is a dense $G_\delta$
subset of $\supp(\mu_y)\times\supp(\mu_y)$.
\end{rem}

\begin{rem} \label{rem:n-chaos}
Let $n\geq 2$.
An $n$-tuple $(x_1,x_2,\dotsc,x_n)\in X^n $ is called
\emph{$(\{s_i\},\delta)$-$n$-scrambled}
if
\[\limsup_{i\to+\infty} \min_{1\leq j<k\leq n} \rho(s_i x_j,s_ix_k)\geq\delta
\quad\text{ and }\quad
\liminf_{i\to+\infty}
\max_{1\leq j<k\leq n} \rho(s_ix_j,s_ix_k)=0.\]
Following ideas in \cite{X05},
we say that a subset $K$ of $X$ is \emph{$(\{s_i\},\delta)$-$n$-scrambled}
if for every  $n$ pairwise distinct points $x_1,x_2,\dotsc,x_n\in K$,
$(x_1,x_2,\dotsc,x_n)$ is $(\{s_i\},\delta)$-$n$-scrambled.

In fact, we can require $K$ be to $(\{s_i\},\delta)$-$n$-scrambled
in Theorem~\ref{thm:Main-result-3}.
As the proof is almost the same as in Section 3,
we only outline the ideas.
Assume $\mu\in \mathcal{M}^e(X,G)$ with $h_\mu(X,G)>0$ and
$\mu=\int  \mu_y d \nu(y)$ be the disintegration of $\mu$ with respect
to the the Pinsker factor  $(Y,\mathcal{P}_\mu(G),\nu,G)$.
For $n\geq 2$, let
\[ \lambda_n=\int  \underbrace{\mu_y\times\mu_y\times\dotsb\times\mu_y}
_{n \text{ times}}d\nu(y)=
\int\mu_y^{(n)}d\nu(y).\]
Similarly as in Lemmas~\ref{lem:lambda} and \ref{lem:relative-UPE},
we have $\lambda_n\in \mathcal{M}^e(X^n,G)$,
$P_{\lambda_{n}}(G)=p^{-1}(\mathcal{P}_\mu(G))$ and
$ \lambda_n=
\int\mu_y^{(n)}d\nu(y)$ can be regard as
the disintegration of $\lambda_n$ with respect to the Pinsker factor of $(X^n,\lambda_n,G)$.

Let
\[
\Delta^{(n)}=\{(x_1,x_2,\dotsc,x_n)\in X^n\colon
\exists 1\leq j<k\leq n,\ \textrm{s.t.}\ x_j=x_k\}
\]
As  $\mu_y$ is non-atomic for $\nu$-a.e.\ $y\in Y$,
 $\lambda_n(\Delta^{(n)})=0$.
For $r>0$, let
\[\Delta^{(n)}_r=\{(x_1,x_2,\dotsc,x_n)\in X^n\colon
\exists 1\leq j<k\leq n,\ \textrm{s.t.}\
\rho(x_j,x_k)<r\}
\]
It is clear that $\Delta^{(n)}=\bigcap_{k=1}^{+\infty} \Delta^{(n)}_{\frac{1}{k}}$.
In the proof of Lemma~\ref{lem:not-asy},
using $\lambda_n$ and  $\Delta^{(n)}_r$ instead of $\lambda$ and $\Delta_r$,
we can show that for every $\eps>0$ there exists a $r>0$
and a measurable subset $D$ of $Y$ with $\nu(D)>1-4\eps$
such that
\[\mu_y^{(n)}\bigl(R_r^{(n)}(\{s_i\})\bigr)=1\]
for all $y\in D$, where
\[R_r^{(n)}(\{s_i\})=\Bigl\{(x_1,x_2,\dotsc,x_n)\in X^n\colon
\limsup_{i\to+\infty} \min_{1\leq j<k\leq n}
\rho(s_i x_j,s_ix_k)\geq r\Bigr\}.\]
Let
\[
\Delta^{n}=\{(x,x,\dotsc,x)\in X^n\colon x\in X
\}
\]
For $r>0$, let
\[\Delta^{n}_r=\{(x_1,x_2,\dotsc,x_n)\in X^n\colon
\rho(x_j,x_k)<r, 1\leq j<k\leq n\}
\]
In the proof of Lemma~\ref{lem:proximal-r},
using $\lambda_n$ and  $\Delta^{n}_r$ instead of $\lambda$ and $\Delta_r$,
we can show that
\[\lambda_n\bigl(P^{(n)}(\{s_i\})\bigr)=1,\]
where
\[P^{(n)}(\{s_i\})=\Bigl\{(x_1,x_2,\dotsc,x_n)\in X^n\colon
\liminf_{i\to+\infty}
\max_{1\leq j<k\leq n} \rho(s_ix_j,s_ix_k)=0
\Bigr\}.\]
\end{rem}

\begin{rem}\label{rem:sofic}
Sofic groups were introduced by Gromov in \cite{G99}
as a common generalization of amenable and residually finite groups.
Initiated in a breakthrough of Bowen in \cite{B10},
a substantial amount of progress has been made in expanding
the entropy theory for actions of discrete amenable groups to sofic groups.
We refer the reader to a recent book \cite{KL16} for this topic.
It is natural to ask that whether a similar result of Theorem~\ref{thm:main-result-1} holds for sofic group actions.
Since our proof depends on Theorem~\ref{thm:entorpy-UPE-extension},
it is not clear how to extend it to sofic group actions.
\end{rem}

\begin{rem}\label{rem:proof}
As we can see, Theorem~\ref{thm:main-result-1} is a purely topological result.
But our proof heavily relies on ergodic theory.
It would be interesting to know whether
there is a topological or combinatorial proof.
In \cite{KL07},  using a local analysis of combinatorial independence
of topological entropy,
Kerr and Li showed that
if an amenable group action has positive topological entropy
then it is Li-Yorke chaotic, see \cite{KL13} for sofic gorup actions.
But it is not clear how to adapt their method to our setting.
\end{rem}

\subsection*{Acknowledgements}
Part of this work was done during a visit of W. Huang and J. Li to
the Chinese University of Hong Kong.
They would like to thank Prof.\ De-Jun Feng for his warm hospitality.
We are grateful to Prof.\ Hanfeng Li  for helpful suggestions,
which lead to the Remarks \ref{rem:n-chaos}, \ref{rem:sofic} and~\ref{rem:proof}.
The authors would also like to thank the referee who made significant comments and fixed many English
errors. This research was supported in part by NNSF of China (11731003, 11771264, 12090012, 12031019) and NSF of Guangdong Province (2018B030306024).

\end{document}